\documentclass{birkjour}
\usepackage{graphics}
\usepackage{epsfig}
\usepackage{amsfonts}
\usepackage{amssymb}
\usepackage{amsthm}
\usepackage{newlfont}
\usepackage{slashbox}
\usepackage{cases}
\usepackage{mathtools}
 \newtheorem{thm}{Theorem}[section]

 \theoremstyle{definition}
 
 \theoremstyle{remark}

 \numberwithin{equation}{section}

\begin{document}

%-------------------------------------------------------------------------
% editorial commands: to be inserted by the editorial office
%
%\firstpage{1} \volume{228} \Copyrightyear{2004} \DOI{003-0001}
%\seriesextraline{This is the Concrete Title of this Book\br H.E. R and S.T.C. W, Eds.}
%\firstpage{1}

\title[Development of a Surface having Regular Polygonal Base and Elliptic Arcs]
 {\begin{center} Development of a Surface having Regular Polygonal Base and Elliptic Arcs\end{center}}

%----------Author 1
\author{Shahid Saeed Siddiqi}
\address{Faculty of Information Technology, University of Central Punjab, Lahore-Pakistan}
\email{shahid.saeed@ucp.edu.pk}
%----------Author 2
\author{Abdul Rauf Nizami}
\address{Faculty of Information Technology, University of Central Punjab, Lahore-Pakistan}
\email{arnizami@ucp.edu.pk}
%----------Author 3

\maketitle
%%% ----------------------------------------------------------------------
%\vspace{20pt}
%\date{today}%February 08, 2013
\begin{abstract}
This paper aims to develop the mathematical representation of a surface generated by elliptical arcs joining the sides of a regular polygon to a point lying vertically upward on the central axis of the polygon. The volume of the corresponding solid has also been determined.
\end{abstract}
%%%---------------------------------------------------------------------
%92E10

%%%------------------------------------------------------------------------
%%% ----------------------------------------------------------------------
%\maketitle

\pagestyle{myheadings}
\markboth{\centerline {\scriptsize
 Shahid S. Siddiqi and A. R. Nizami}} {\centerline {\scriptsize
 Development of a Surface having Regular Polygonal Base and Elliptic Arcs}}
%%% ----------------------------------------------------------------------
%\tableofcontents
%%----------------------------------------------- section --------------------------------------------------------
%%----------------------------------------------- section --------------------------------------------------------
%\tableofcontents
\section{Introduction}To find the volume of hemisphere, two different approaches are available, integral  and geometric. In integral approach, the volume of hemisphere is calculated using single and double integrals \cite{Anton 10, Thomas 13}. Geometrically, the volume of hemisphere is determined by showing the volume of hemisphere equal to the volume of the cylinder outside the cone (as shown in Figure~\ref{fig0}) \cite{Roberts}. The idea was actually the cross-sections have the same area.
\begin{figure}[h]
  \centering
     \includegraphics[width=8cm]{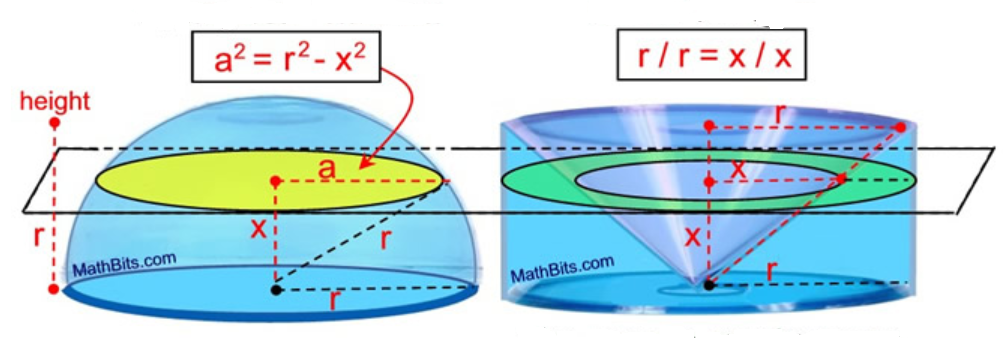}
     \caption{}\label{fig0}
\end{figure}

\noindent In this paper, the hemisphere has been developed using a different approach. The cylinder without the  cone inside it has been sliced into $n$ washers. These washers, after converting to discs of volumes of corresponding washers, have been stacked  as shown in   Figure~\ref{fig1}.
\pagebreak
\goodbreak
\begin{figure}[h]
  %\centering
  \begin{minipage}{3.5cm}
   \includegraphics[width=3.5cm]{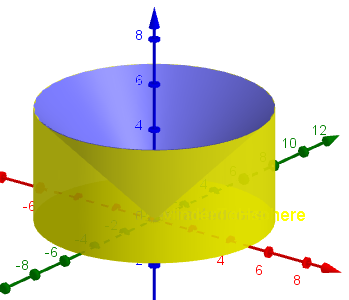}
  \end{minipage}
  \begin{minipage}{3.5cm}
   \includegraphics[width=3.5cm]{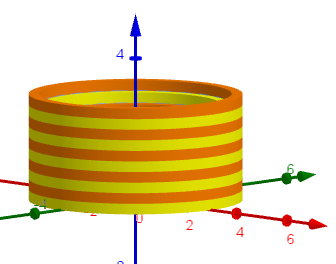}
  \end{minipage}
 \begin{minipage}{3.5cm}
   \includegraphics[width=3.5cm]{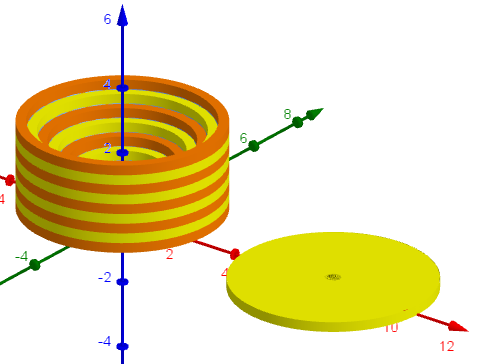}
  \end{minipage}
  \end{figure}

\begin{figure}[h]
%  \centering
  \begin{minipage}{3.5cm}
   \includegraphics[width=3.5cm]{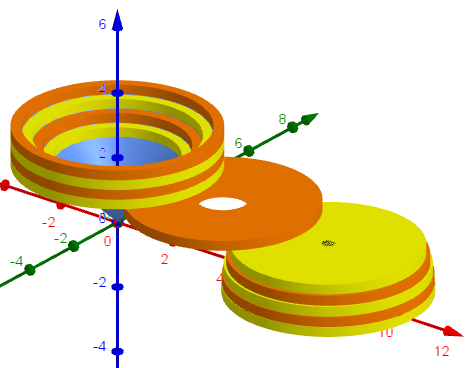}
  \end{minipage}
  \begin{minipage}{3.5cm}
   \includegraphics[width=3.5cm]{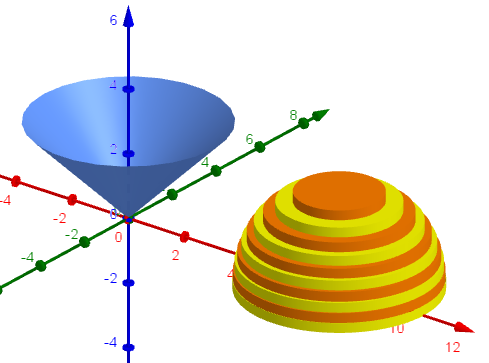}
  \end{minipage}
 \begin{minipage}{3.5cm}
   \includegraphics[width=3.5cm]{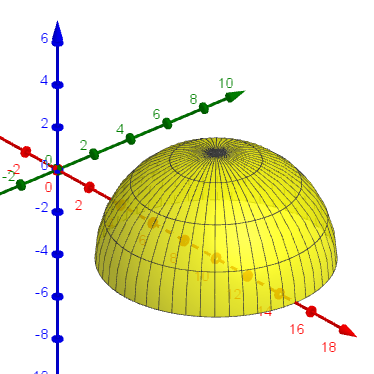}
  \end{minipage}
    \caption{}\label{fig1}
\end{figure}
\noindent The hemisphere is thus obtained as a limiting surface as $n\rightarrow \infty$. The idea of obtaining hemisphere has been extended to generate a limiting surface obtained by stacking the slices of cuboid, excluding the pyramid inside it, after converting the slices into the slabs of volumes of corresponding slices. Lastly, the idea has been extended to obtain the surface with any regular polygonal base.

\noindent To develop the surface having a regular polygonal base along with the surface generated by the elliptic arcs joining the sides of the polygon to the point lying vertically upward on the central axis of the polygon, consider the solid corresponding to square base, as shown in Figure~\ref{fig2}.
 \begin{figure}[h]
  \centering
  \includegraphics[width=6.5cm,height=3.3cm]{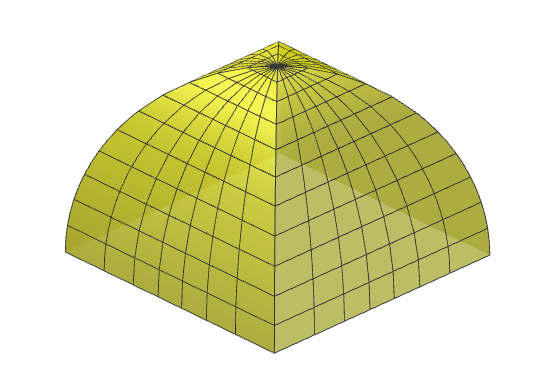}\\
  \caption{Main Solid}\label{fig2}
\end{figure}\\
Initially, the geometrical development and mathematical representation of the case with square base has been discussed in Sections ~\ref{sec2} and 3, respectively. The case, with the regular polygonal base, has been developed in Section 4 .
\section{Geometrical Development of the Surface}\label{sec2}
 To develop the structure shown in Figure~\ref{fig2}, consider a pyramid with square base lying in the cuboid with the same square base and height equal to half of the side of the square, as shown in Figure~\ref{fig3}.
\begin{figure}[h]
  \begin{center}
      \begin{minipage}{5.0cm}
   \includegraphics[width=4.0cm]{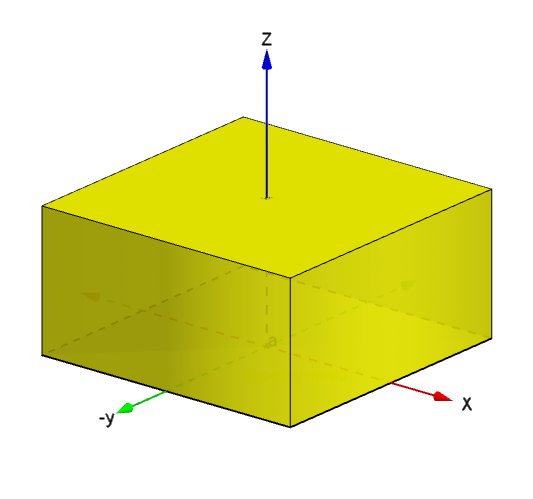}
     \end{minipage}
  \begin{minipage}{5.0cm}
   \includegraphics[width=4.0cm]{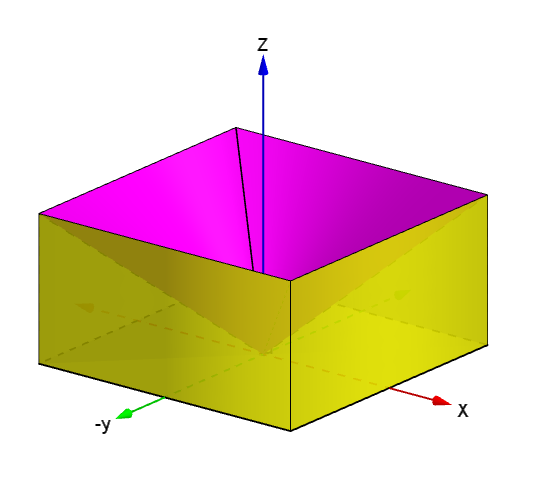}
  \end{minipage}
  \end{center}
       \caption{Pyramid in Cuboid}\label{fig3}
\end{figure}

\noindent The cuboid with inserted pyramid has been sliced  into $n$ equal parts parallel to the $xy$-plane, as shown in Figure~\ref{fig4}.
%\pagebreak
\begin{figure}[h]
  \centering
      \includegraphics[width=8.0cm]{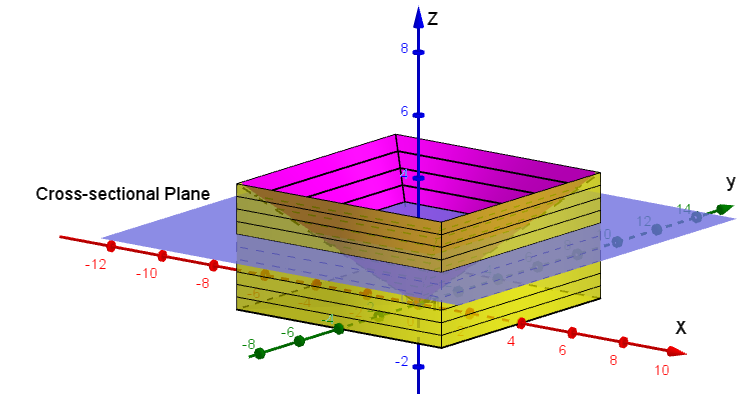}
  \caption{Sliced Cuboid}\label{fig4}
\end{figure}

\noindent Moreover, each slice has been transformed to slab having the volume equal to the volume of the original slice. One of the slices has been shown in Figure~\ref{fig5}.
\begin{figure}[h]
  \centering
%  % Requires \usepackage{graphicx}
  \includegraphics[width=5cm]{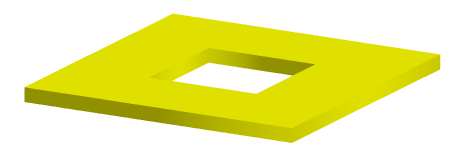}
  \includegraphics[width=4cm]{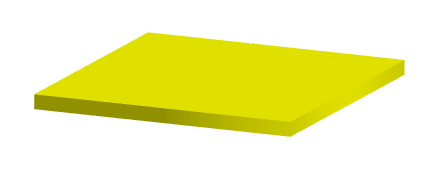}\\
  \caption{Slice with corresponding slab}\label{fig5}
\end{figure}

\noindent Also, all the new slabs have been concentrically stacked with the same order, as shown in Figure~\ref{fig6}.
  \begin{figure}[h]
  \centering
%  % Requires \usepackage{graphicx}
  \includegraphics[width=6.5cm]{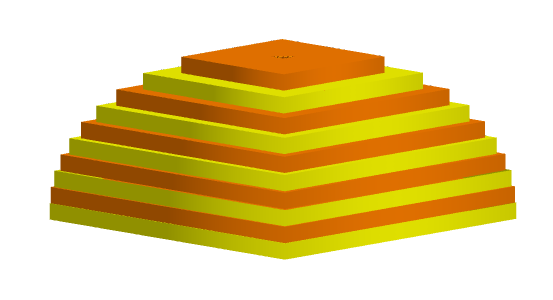}\\
  \caption{Stacked Slabs}\label{fig6}
\end{figure}

\noindent  It may be noted that increasing the number of slices will smooth the surface generated by these slabs. In the limiting case, as $n \rightarrow \infty$, the  generated surface will be perfectly smooth as shown in  Figure~ \ref{fig7}
  \begin{figure}[h]
  \centering
%  % Requires \usepackage{graphicx}
  \includegraphics[width=7.0cm]{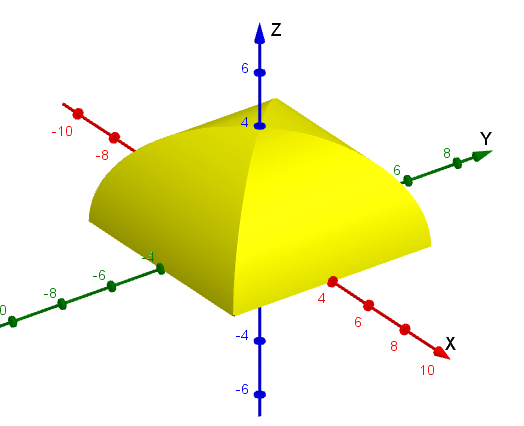}\\
  \caption{Smooth Surface}\label{fig7}
\end{figure}
\section{Mathematical Development of the Surface}
\begin{thm}
The curve of intersection of the surface and a vertical plane through the origin is elliptic.
\end{thm}
\begin{proof}
Without loss of generality, consider the curve obtained by intersection of the surface and the vertical plane (passing through the diagonal of the square), as shown in  Figure~ \ref{fig8}.
\pagebreak
\goodbreak
\begin{figure}[h]
  \centering
%  % Requires \usepackage{graphicx}
  \includegraphics[width=5.0cm]{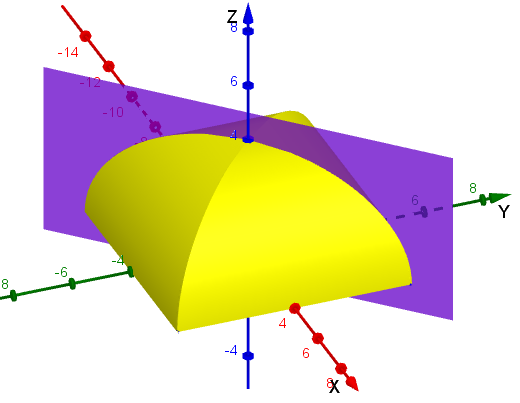}\\
  \caption{Intersection of the surface with the vertical plane }\label{fig8}
\end{figure}

\noindent Since the surface has been obtained as $n \rightarrow \infty$ of the $n$ stacked slabs, consider the intersection of vertical plane passing through the diagonal of the square with the stacked slabs, along with  the base of the $(i+1)$st slab and a point $Q(x,y,z)$,as shown in Figure~\ref{fig9}.
 \begin{figure}[h]
  \centering
%  % Requires \usepackage{graphicx}
  %\includegraphics[width=3.5cm]{figure12point1}\\
 % \includegraphics[width=3.5cm]{figure12point2}\\
 \includegraphics[width=6.5cm]{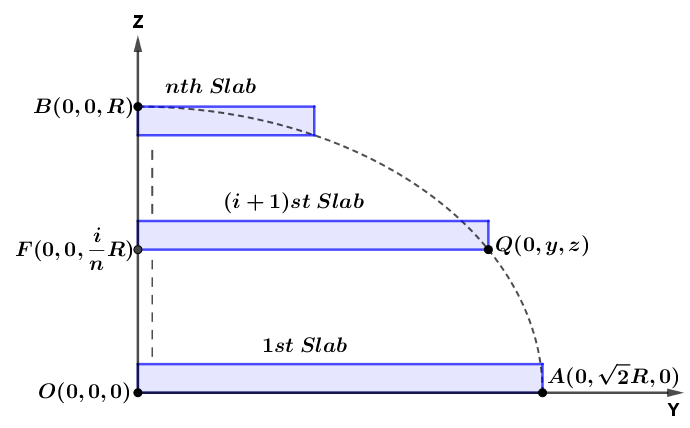}\\
  \caption{ }\label{fig9}
\end{figure}

\noindent To establish the relation among the coordinates of the point Q, the resulting intersection has been rotated about origin in $xy$-plane through the angle $\frac{\pi}{4}$, resulting the point Q(0,y,z), as shown in Figure~\ref {fig9}. It can, thus, be written as
\begin{eqnarray*}
  z &=& \frac{i}{n}R \\
  y &=& \sqrt{(\sqrt{2}R)^{2}-(\frac{i}{n}\times\sqrt{2}R)^{2}}=\sqrt{2}R\sqrt{1-(\frac{i}{n})^{2}}\\
 y^{2} &=& 2R^{2}(1-(\frac{i}{n})^{2})=2R^{2}(1-(\frac{z}{R})^{2})=2R^{2}-2z^{2}\\
 y^{2}+z^{2} &=& 2R^{2}\\
 \frac{y^{2}}{(\sqrt{2}R)^{2}}+\frac{z^{2}}{(R)^{2}} &=& 1
\end{eqnarray*}
Since the last equation represents the ellipse in the $yz-$ plane, the arcs are elliptic.
\end{proof}
\section{Development of the Surface}A square $A_{1}A_{2}A_{3}A_{4}$ has been considered in the $xy$-plane with side $2R$. The center of the square has been taken at the origin and  its side $A_{1}A_{2}$ perpendicular to the $x$-axis, as shown in Figure~\ref{fig10}.
 \begin{figure}[h]
  \centering
  \includegraphics[width=6.5cm]{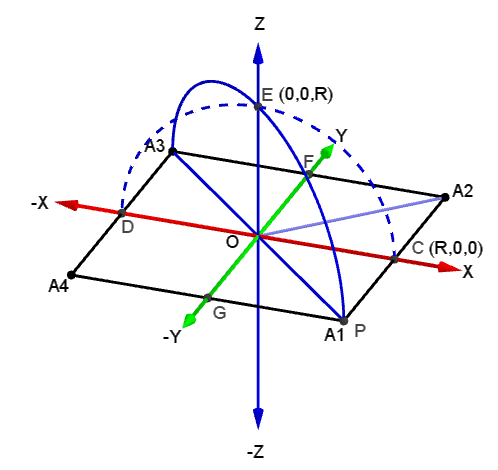}
    \caption{}\label{fig10}
\end{figure}\\
 If $(x,y,z)$ is any point on the quarter-circular arc $CE$, then the parametric equations of this arc are
\[
\left.
\begin{array}{lll}
  x &=& R\cos (t) \\
  y &=& 0\\
  z &=& R\sin (t)
  \end{array}
  \right\},
  t\in \left[ 0,\frac{\pi}{2} \right]
   \] \\
Although the surface can be obtained by rotating the arc $CED$ through an angle $\pi$, the suqare is divided into four triangles $OA_1A_2$, $OA_2A_3$,  $OA_3A_4$  and $OA_4A_1$. The purpose is to parametrize the surface. To generate the surface over the triangle $OA_1A_2$ (shown in Figure~\ref{fig11}), a radial segment $OP$ is swept counterclockwise from $OA_1$ to $OA_2$ such that the point $P$ moves on the segment $A_1A_2$. It follows that the length of $OP$ depends on the angle $r$ the $OP$ makes with the positive $x$-axis.
\begin{figure}[h]
  \centering
  \includegraphics[width=6cm]{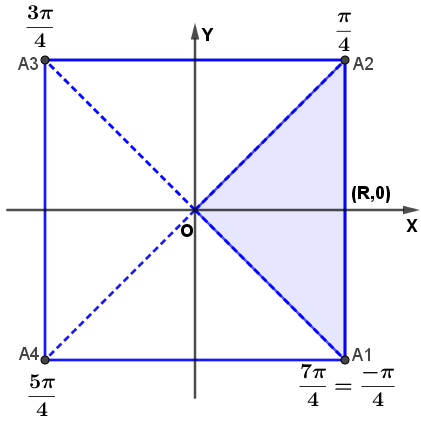}
    %[width=4.0cm]{figure1point7}\\
  \caption{}\label{fig11}
\end{figure}

\noindent To generate the surface over the triangle $OA_2A_3$, the radial segment $OP$ will be rotated from $\frac{\pi}{4}$ to $\frac{3\pi}{4}$. However, the triangle $OA_2A_3$ is considered to be rotated back to fit on the triangle $OA_1A_2$. The purpose is to use the cosine of the angle that $OP$ makes with the positive $x$-axis. This actually gives  the opportunity to control the parametrization of the whole surface in a single formula. Similarly, the surfaces over  the remaining triangles are obtained. \\

\noindent The parametrization of the square has been calculated as
\[   \left.
\begin{array}{lll}
      x &=& \frac{R}{a}\cos(r) \\
      & & \\
  y &=& \frac{R}{a}\sin(r) \\
  \end{array}
\right\},
 r\in \left[ \frac{-\pi}{4}, \frac{7\pi}{4}\right]  \]
$
a= \cos \left(r-(i-1)\frac{\pi}{2} \right), r\in \left[ \left( \frac{-1}{2}+(i-1)\right)\frac{ \pi}{2},\left( \frac{-1}{2}+i \right)\frac{ \pi}{2}\right],i=1,2,3,4$.\\

\noindent The above rotations can be performed using the following matrix transformation.

$$
\left(
          \begin{array}{c}
            x \\
            y \\
            z \\
          \end{array}
        \right) =
\left(
   \begin{array}{ccc}
     \cos(r) & -\sin(r) & 0 \\
     \sin(r) & \cos(r) & 0 \\
     0 & 0 & 1 \\
   \end{array}
 \right) \left(
          \begin{array}{c}
            \frac{R}{a}\cos(t) \\
            0\\
            R \sin(t) \\
          \end{array}
        \right)
$$ \\

\noindent The required surface can, thus, be expressed in the following parametrization\\
\[ \left.
\begin{array}{lll}
      x &=& \frac{R}{a}\cos(r)\cos(t) \\
      & & \\
      y &=& \frac{R}{a}\sin(r)\cos(t) \\
  & & \\
      z &=& R\sin(t)
\end{array}
\right\}, r \in \Big [ \frac{-\pi}{4}, \frac{7 \pi}{4} \Big ], t\in[0,\frac{\pi}{2}]
 \]
$ a= \cos \left(r-(i-1)\frac{\pi}{2} \right), r\in \left[ \left( \frac{-1}{2}+(i-1)\right)\frac{ \pi}{2},\left( \frac{-1}{2}+i \right)\frac{ \pi}{2}\right],i=1,2,3,4$.
\section{Generalization}
Extending the surface with square base to a surface with any regular $n$-gonal base, the general result has been obtained as:
\begin{thm}
The parametric representation of the surface with any regular $n$-gonal base is
\[   \left.
\begin{array}{lll}
      x &=& \frac{R}{a}\cos(r)\cos(t) \\
      & & \\
  y &=& \frac{R}{a}\sin(r)\cos(t) \\
  & & \\
  z &=& R\sin(t)  \\
\end{array}
\right\},
 r\in \left[ \left( \frac{-1}{2}\right)\frac{2 \pi}{n},\left( \frac{-1}{2}+n \right)\frac{2 \pi}{n}\right], t\in \left[ 0,\frac{\pi}{2} \right],  \]
$
a= \cos \left(r-(i-1)\frac{2 \pi}{n} \right), r\in \left[ \left( \frac{-1}{2}+(i-1)\right)\frac{2 \pi}{n},\left( \frac{-1}{2}+i \right)\frac{2 \pi}{n}\right], i=1,2,3, \ldots ,n
$
\end{thm}
\begin{proof}
The procedure is already explained in case of a squared base. However, in the generalized case quarter circular arc will be rotated on the complete boundary of the polygon, instead of semicircular arc rotated only on the half boundary of the square. For instance, the surfaces for $ n=5, 6$ and $7$ have been shown in Figure~\ref{fig12}.\\
\begin{figure}[h]
  \centering
  \begin{minipage}{3.75cm}
  \includegraphics[width=3.5cm]{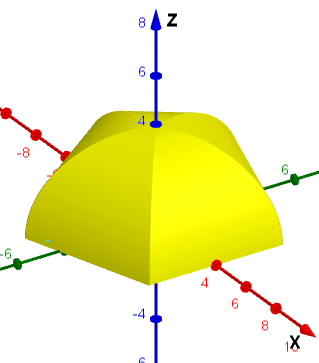}
  \end{minipage}
  \begin{minipage}{3.75cm}
  \includegraphics[width=3.5cm]{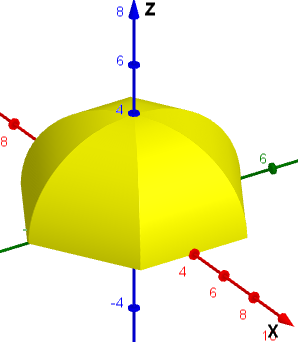}
  \end{minipage}
  \begin{minipage}{3.75cm}
  \includegraphics[width=3.5cm]{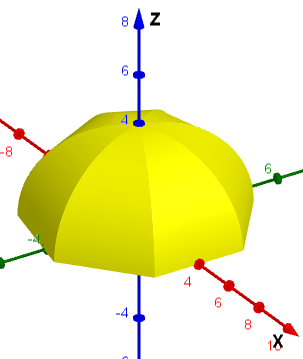}
  \end{minipage}
  \caption{}\label{fig12}
   \end{figure}\\
\end{proof}
%\goodbreak
%\newpage

\section{Volume of the Solid}
First the volume of solid with square base is calculated. \\
Since the side of the cuboid is $2R$ with height $R$, the volume of the cuboid is $V_c = (4R^2)R=4R^3$ and the volume of the pyramid is $V_p = (\frac{1}{3})(V_c)$. Thus the volume of the corresponding solid is $V = V_c-V_p = (\frac{2}{3})(V_c) = (\frac{8}{3})R^3$ \\
To determine the volume of solid corresponding to  regular  $n$-sided polygonal base, each $n$-sided regular polygon has been partitioned into $n$ equilateral triangles, each with height $R$; an equilateral triangle is depicted in Figure~\ref{fig13}, for $n=7$.\\
\begin{figure}[h]
  \centering
  \includegraphics[width=6.5cm]{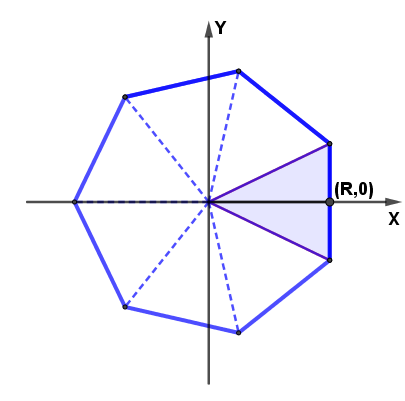}
    \caption{}\label{fig13}
   \end{figure}\\
 The area  of each equilateral triangle is determined as $ R^2 \tan(\frac{\pi}{n})$,  which shows that the area  of $n$-sided polygon will be $ n R^2 \tan(\frac{\pi}{n})$.\\
The volume $V_{Prism}$ of the corresponding prism with height $R$ will be\\
$ V_{Prism} = A_nR = n R^2 \tan(\frac{\pi}{n})R = n R^3 \tan(\frac{\pi}{n})$ \ which shows that the volume of corresponding pyramid will be $V_{Pyramid} = \frac{1}{3}V_{Prism}$\\
Thus the volume of solid can be calculated as\\
$V = V_{Prism}-V_{Pyramid}=\frac{2}{3} V_{Prism} = \frac{2}{3}n R^3 \tan(\frac{\pi}{n})$

\section{Conclusion}

In this article a surface has been geometrically developed and its parametrization has also been given. A solid prism with  a regular polygon base and hight equal to the radius of the in-circle of the base, contains a solid cone  with base as top of the prism and vertex at the center of the prism base. The cone has been then pulled out of the cylinder, and the cylinder has been cut into $n$ horizontal slices (washers). Corresponding to each washer, a disc with volume equal to the volume of the washer has been considered. Then the discs have been stacked in the sequence of the corresponding washers. By indefinitely increasing the number of washers,  the  solid with smooth surface has been  obtained. Secondly,  the curve of intersection of the surface and any vertical plane through the origin has been proved elliptic. Finally, parameterization of both the regular polygonal base and the developed surface, have been determined. Moreover, following this technique the volume of the solid has been determined.

\end{document}